\documentclass[11pt]{amsart}

\usepackage[top=30mm, bottom=30mm, left=30mm, right=30mm]{geometry}  

\usepackage{mathpazo}
\usepackage{amsmath,amssymb, amscd, color}
\usepackage{graphicx}
\usepackage{import}
\usepackage{amscd}
\usepackage{wrapfig}
\usepackage{epsfig}
\numberwithin{equation}{section}
\usepackage[alphabetic,nobysame]{amsrefs}
\usepackage{float}
\usepackage{tikz}
\usetikzlibrary{positioning}
\usepackage{comment}

\DeclareMathOperator{\WP}{WP}
\DeclareMathOperator{\hyp}{hyp}
\DeclareMathOperator{\thick}{Thick}

\newtheorem{theorem}{Theorem}[section]

\newtheorem{lemma}[theorem]{Lemma}
\newtheorem{corollary}[theorem]{Corollary}
\newtheorem{proposition}[theorem]{Proposition}

\newcommand{\teichmuller}{Teichm{\"u}ller{ }}

\makeatletter
 \let\c@theorem=\c@subsection
 \let\c@conjecture=\c@subsection
 \let\c@lemma=\c@subsection
 \let\c@proposition=\c@subsection
 \let\c@claim=\c@subsection
 \let\c@question=\c@subsection
 \let\c@criterion=\c@subsection
 \let\c@vfconj=\c@subsection
 \let\c@definition=\c@subsection
 \let\c@notation=\c@subsection
 \let\c@remark=\c@subsection
 \let\c@example=\c@subsection
 \let\c@equation=\c@subsection
 \let\c@figure=\c@subsection
 \let\c@wrapfigure=\c@subsection

\makeatother

\begin{document}
\title{Weil--Petersson geodesics on the modular surface}

\author[Gadre]{Vaibhav Gadre}
\address{\hskip-\parindent
     School of Mathematics and Statistics\\
     University of Glasgow\\
     University Place\\
      Glasgow\\
      G12 8QQ United Kingdom}
\email{Vaibhav.Gadre@glasgow.ac.uk}

 


\begin{abstract}
We consider the Weil--Petersson (WP) metric on the modular surface.
We lift WP geodesics to the universal cover of the modular surface and analyse geometric properties of the lifts as paths in the hyperbolic metric on the universal cover. 
For any pair of distinct points in the thick part of the universal cover, we prove that the WP and hyperbolic geodesic segments that connect the pair, fellow-travel in the thick part and all deviations between these segments arise during cusp excursions. 
Furthermore, we give a quantitative analysis of the deviation during an excursion. 

We leverage the fellow traveling to derive a correspondence between recurrent WP and hyperbolic geodesic rays from a base-point.
We show that the correspondence can be promoted to a homeomorphism on the circle of directions. 
By analysing cuspidal winding of a typical WP geodesic ray, we show that the homeomorphism pushes forward a Lebesgue measure on the circle to a singular measure.
In terms of continued fraction coefficients, the singularity boils down to a comparison that we prove, namely, the average coefficient is bounded along a typical WP ray but unbounded along a typical hyperbolic ray.
\end{abstract}


\maketitle


\section{Introduction}
We consider the Weil--Petersson metric on the modular surface $S$. 
Lifting a WP geodesic segment to the universal cover of $S$, we analyse its geometric properties in the hyperbolic metric on the universal cover. 
The thick part of the universal cover is the complement of lifts of cusp neighbourhoods. 
For any pair of distinct points in the thick part, we prove that the WP and hyperbolic geodesic segments connecting these points fellow travel in the thick part with all deviations between them arising during excursions in cusp neighbourhoods. 
Furthermore, we give a quantitative analysis of these deviations during an excursion. 
Our results should extend to mildly constrained general WP-type metrics.
For this reason, we keep our discussions on a more general footing. 

We now provide a more precise description.
For notation, we denote $S$ with the WP metric by $X$ and $S$ with the hyperbolic metric by $Y$. 
We will denote the universal cover of $S$ by $\mathbb{D}$ and the lift of the metrics by $d_{\WP}$ and $d_{\hyp}$ respectively.

Let $\delta$ denote the distance to the cusp in $X$. 
For some $B > 0$ small enough, we may define a cusp neighbourhood by the condition $\delta < B$. 
We call the boundary locus $\delta < B$, a WP-horocycle.
When $B$ is small enough the cusp neighbourhood $N = \{ \delta < B\}$ is convex. 

The convexity of $N$ and the compactness of $X \setminus N$ have the following simple consequence: there exists a smaller neighbourhood $\{ \delta < b \}$ where $0 < b < B$ such that for any pair of points $p, q$ in $X \setminus N$, the shortest WP geodesic connecting them is disjoint from $\{ \delta < b \}$. 

The consequence stated above implies that a lift of $N$ to the universal cover $\mathbb{D}$ is also quasi-convex in the hyperbolic metric on $\mathbb{D}$. 
So a lift of $N$ to $\mathbb{D}$ is approximated by a horoball in $d_{\hyp}$.
In Section \ref{s.fellow}, we will clarify in a more precise technical sense how the hyperbolic horocycle bounding the largest horoball contained in a lift of $N$ approximates the WP horocycle. 
The $\pi_1(S)$-orbit of the horoball  gives us a collection $\mathcal{H}$ of disjoint horoballs in $\mathbb{D}$. 
We shall call the region in $\mathbb{D}$ that is complementary to the collection of these horoballs the \emph{thick part} of $\mathbb{D}$. 
Note that quotient of the thick part by $\pi_(S)$ is thick in either metric, that is, it is disjoint from a cusp neighbourhood in both $d_{\WP}$ and $d_{\hyp}$. 

Let $p, q$ be a pair of distinct points in the thick part. 
Let $\gamma_{\WP}[p,q]$ and $\gamma_{\hyp}[p,q]$ be respectively the WP and hyperbolic geodesic segments between $p$ and $q$.

A \emph{thick-thin decomposition} of the WP segment $\gamma_{\WP}[p,q]$ is a partition of $\gamma_{\WP}[p,q]$ into WP geodesic sub-segments as
\begin{equation}\label{e.WP-decompose}
\gamma_{\WP}[p,q] = [y_0, y_1] \cup [y_1, y_2] \cup \dots \cup [y_{2k-1}, y_{2k}] \cup [y_{2k}, y_{2k+1}],
\end{equation}
where 
\begin{itemize}
\item $k \geqslant 0$,
\item $y_0 = p$, $y_{2k+1} = q$, 
\item the segments $[y_{2r}, y_{2r+1}]$ for $r = 0, k$ are contained in the thick part, and
\item the open segments $(y_{2r-1}, y_{2r})$ for $r = 1, \cdots, k$ if non-empty, are contained in a horoball in $\mathcal{H}$.
\end{itemize}
Analogously, a thick-thin decomposition of $\gamma_{\hyp}[p,q]$ is a partition of $\gamma_{\hyp}[p,q]$ into hyperbolic geodesic sub-segments as 
\begin{equation}\label{e.hyp-decompose}
\gamma_{\hyp}[p,q] = [z_0, z_1] \cup [z_1, z_2] \cup \dots \cup [z_{2j-1}, z_{2j}] \cup [z_{2j}, z_{2j+1}],
\end{equation}
where 
\begin{itemize}
\item $j \geqslant 0$, 
\item $z_0 = p$, $z_{2j+1} = q$ and 
\item the segments $[z_{2r}, z_{2r+1}]$ for $r = 0, j$ are contained in the thick part, and
\item the open segments $(z_{2r-1}, z_{2r})$ for $r = 1, \cdots, j$, if non-empty, are contained in a horoball in $\mathcal{H}$ 
\end{itemize} 

We prove the following theorem.

\begin{theorem}\label{t.main}
There is a constant $R> 0$ such that for any pair of distinct points $p,q$ in the thick part of $\mathbb{D}$, there exist thick-thin decompositions (\ref{e.WP-decompose}) and (\ref{e.hyp-decompose}) of $\gamma_{\WP}[p,q]$ and $\gamma_{\hyp}[p,q]$ respectively, such that  $j =k$ and the segments $[z_{2r},z_{2r+1}]$ $R$-fellow travel (in either metric) the segments $[y_{2r}, y_{2r+1}]$ for $r = 0, \cdots, k$.
\end{theorem}

We will also estimate the deviation in $d_{\hyp}$ between the geodesics during each cusp excursion, that is, we will estimate the hyperbolic distance between corresponding points on the open segments $(z_{2r-1}, z_{2r})$ and $(y_{2r-1}, y_{2r})$ when they are both non-empty and large enough.

The coarse geometric viewpoint offers a clarifying perspective of Theorem \ref{t.main}.
Recall that the modular surface is the moduli space of hyperbolic once-punctured tori. 
The WP metric completion of $\mathbb{D}$ includes noded surfaces that arise by pinching the systole on a marked once-punctured torus.
It follows that the WP completion is quasi-isometric to the Farey graph with every edge having length one.
In either metric, the thick part of the universal cover $\mathbb{D}$ is quasi-isometric to the tree dual to the Farey triangulation. 
A geodesic segment in the tree can be written as a finite sequence of alternating rights and lefts. 
Any pair of triangles that differ only by rights (or only by lefts) share a vertex in the Farey graph. 
When the number of consecutive rights (or lefts) is more than two, the shortest path in the Farey graph has length two and passes through the shared vertex. 
In other words, the WP metric has a shortcut that passes through the noded surface corresponding to the shared vertex. 
The actual WP geodesic follows this shortcut closely by doing a cusp excursion winding around the noded surface.
This feature is shared by the hyperbolic metric except that the cusps are infinitely far away.
The actual hyperbolic geodesic achieves its shortcut by a cusp excursion of hyperbolic length that is logarithmic in the number of consecutive rights (or lefts). 

\subsection{The circle map:} 
Let $p \in \mathbb{D}$ be a base-point. 
Passing to a smaller WP cusp neighbourhood if required, we may assume that $p$ is in the thick part. 
Consider a WP ray $\gamma$ from $p$ that recurs to the thick part. 
By using Theorem~\ref{t.main} along recurrence times $\gamma_t$, we get hyperbolic geodesic segments $\gamma_{\hyp}[p , \gamma_t]$. 
We will show that we can pass to a limit along these hyperbolic geodesic segments to associate a hyperbolic ray $\gamma'$ from $p$ which fellow travels $\gamma$ in the thick part and all deviations arise during cusp excursions. 
In particular, the ray $\gamma'$ is also recurrent to the thick part.

We may identify $S^1$ with the unit tangent circles $T_p^1 (\mathbb{D}, d_{\WP})$ and $T_p^1 (\mathbb{D}, d_{\hyp})$ in the respective metrics.
The limiting argument on recurrent directions sketched above defines a map $\psi$ from a dense subset of $S^1$ to $S^1$. 
In the reverse direction, we can apply a similar limiting argument to associate a recurrent WP ray to every recurrent hyperbolic ray.  
This defines a circle map $\xi$ from a dense subset of $S^1$ to $S^1$. 
It also follows that the compositions $\psi \circ \phi$ and $\phi \circ \psi$ are always defined and $\phi \circ \psi = \text{id}$ and $\psi \circ \phi = \text{id}$.
Thus $\psi$ is a bijection from the dense set of recurrent WP-directions to the dense set of hyperbolic recurrent directions.
We will further show that $\psi$ and $\phi$ are monotone.
From the density of the domains and the monotonicity, we may conclude that $\psi$ and $\phi$ extend to homeomorphisms of $S^1$ that are inverses of each other. 

Let $\ell$ be a Lebesgue measure on $S^1$.
The Liouville measures $\mu_{\WP}$ or $\mu_{\hyp}$ are absolutely continuous with respect to each other.
After identifying $S^1$ with either of the unit tangent circles $T_p^1 (\mathbb{D}, d_{\WP})$ and $T_p^1 (\mathbb{D}, d_{\hyp})$, we may consider the measure $\ell$ to be the conditional measure on the circle of directions obtained from either of these Liouville measures. 
The WP flow is ergodic \cite{Pol-Wei}, and in fact, exponentially mixing \cite{BMMW}. Thus, after identifying $T_p^1 (\mathbb{D}, d_{\WP})$ with $S^1$, it follows that $\ell$-almost every $v \in S^1$ determines a recurrent WP ray. 
Thus the map $\psi$ is well-defined on a full measure subset of $S^1$. 

We prove:

\begin{theorem}\label{t.singular}
The push-forward measure $\psi_\ast(\ell)$ is singular with respect to $\ell$.
\end{theorem}
 
Theorem \ref{t.singular} boils down to a comparison of the statistics of cuspidal winding numbers along typical WP and hyperbolic rays. 
In Section \ref{s.random}, we will reformulate this comparison in terms of continued fraction coefficients along typical WP and hyperbolic rays, namely, we will show that the \texttt{"}average\texttt{"} coefficient is bounded along a typical WP ray but unbounded along a typical hyperbolic ray.

\subsection{WP-type metrics on finite type surfaces:}

We now mention the constraints under which the results presented here should extend to general WP-type metrics on surfaces of finite type. 
The first requirement is that cusp neighbourhoods can be chosen to be convex and their horocycles closely approximate the corresponding horocycles for the surfaces of revolution that model them. 
The precise nature of the approximation required is discussed in Section \ref{s.horo} and we discuss why it holds for the modular surface.

For a cusp modelled on the surface of revolution for $y = x^r \, : \, r  \geqslant 3$, the convexity of a small enough cusp neighbourhood translates to a constraint on the lower order terms in the expansion of the metric near the cusp with the metric on a surface of revolution as the leading term. 
See \cite[Remark on page 242]{BMMW} for the details. 

Similarly, the required properties for WP horocycles mentioned in Section \ref{s.horo} should also be derivable once there is a reasonable constraint on the lower order terms in the expansion of the cusp metric. 
We leave the details of this derivation to the interested reader.

\subsection{The WP metric on non-exceptional moduli spaces:}

The analogous passage between WP and \teichmuller geodesics on non-exceptional moduli spaces is likely to be much more involved. 
For example, in the exceptional case the existence of the limiting ray relies on the curvatures being pinched, that is, negative and bounded away from zero.
This is no longer true for non-exceptional moduli. 

Brock-Masur-Minsky associated to ending laminations to WP rays. 
See \cite[Definition 2.5]{BMM}. 
It then turns out that there exist recurrent WP rays with non-uniquely ergodic ending laminations.
See \cite[Theorem 1]{Bro-Mod}. 
As such, a naive limiting argument as we have above cannot hold. 
It is possible however typical WP rays do admit a more subtle limiting argument.

\subsection{Acknowledgements:} I am extremely grateful to Carlos Matheus for the extensive clarifying discussions that I had with him during the course of this work. I would also like to thank Scott Wolpert for his helpful comments on an earlier draft of the paper.

\section{Fellow traveling in the thick part}\label{s.fellow} 

\subsection{WP and hyperbolic horocycles in model co-ordinates:}\label{s.horo}

To compare the two metrics near the cusp, it will be convenient to work in the upper half space co-ordinates $\{ z = x +i y) \, : \, y > 0\}$ on $\mathbb{D}$ with the cusp at infinity.
The Poincare metric on the upper half-space descends to the hyperbolic metric on the modular surface. 
As such, a hyperbolic cusp neighbourhood in these co-ordinates can be chosen to be of the form $\{ y > \text{constant} \}$.

Viewed as the \teichmuller space of hyperbolic once-punctured tori, the upper half-space co-ordinates are closely related to the Fenchel-Nielsen co-ordinates. 
For a marked hyperbolic once-punctured torus $\tau \in S$, let $\ell = \ell_\tau$ be the hyperbolic length of its systole.
Let $t = t_\tau$ be the twist parameter about the systole.
In \cite[Corollary 4.10]{Wol}, Wolpert showed that the WP distance $\delta(\tau)$ of $\tau$ to the cusp of $S$ satisfies
\begin{equation}\label{e.distance} 
\delta = \sqrt{2\pi \ell} + O (\ell^{5/2}).
\end{equation} 

To keep the model expressions for the WP metric simple, we will allow our choice of the co-ordinates $(x,y)$ to be flexible up to a homothety of the form $z \to \alpha z $ for $\alpha \in \mathbb{R}_{> 0}$.
If $B$ is small enough, the homothety constants we will need to choose are uniform. 
This means that our estimates will change only up to multiplicative constants that depend on $B$, which is small enough and fixed once and for all.

Up to such a uniform homothety, the co-ordinates $(x,y)$ near $\tau$ written in terms of the Fenchel--Nielsen co-ordinates $(\ell, t)$ have the following form:
\begin{equation}\label{e.FN}
\ell =  \frac{1}{y} + h(x,y), \hskip 15pt t =  \frac{x}{y} + h'(x,y).
\end{equation} 
where $h(x,y)$ and $h'(x,y)$ comprise of lower order terms.
In particular, we may write $h(x,y)$ as $h(x,y) = g(x,y)/y^2$, where we may assume that both $g$ and its first partials are bounded.
We will use expressions in (\ref{e.FN}) to pass from Fenchel--Nielsen co-ordinates to the upper half space co-ordinates $(x,y)$. 

From equation (\ref{e.distance}) above, a WP horocycle $\{ \delta = B \}$ corresponds to $y = 2 \pi/ B^2$ to the first order.  
Up to uniform homothety, this is the same as $y = 1/B^2$. 
As
\[
\frac{\partial \delta}{ \partial y} = \left( \sqrt{\frac{\pi}{2 \ell}} + O(\ell^{3/2}) \right) \frac{\partial \ell}{\partial y}
\]
if $B$ is small enough then so is $\ell$ and hence $\partial \ell / \partial y \neq 0$.  
By the implicit function theorem, the horocycle $\{ \delta = B \}$ can be written as $(x, f_B(x))$. 
Now the derivative $df_B/ dx = - (\partial \delta / \partial x )/ (\partial \delta / \partial y) = - (\partial \ell/ \partial x)/ (\partial \ell / \partial y)$. By direct computation we get
\[
\frac{\partial \ell / \partial x}{ \partial \ell/ \partial y} = \frac{(1/y^2) \partial g / \partial x}{ (-1/y^2)(1 - \partial g/ \partial y + 2 g / y) } = -\frac{\partial g/ \partial x}{1 - \partial g / \partial y + 2 g/ y} 
\]
As $B$ and hence $y$ is small, it follows that the right hand side is bounded. 
We thus conclude that when the horocycle $\{ \delta = B \}$ is viewed as a graph of a function over $x$, the derivative at any point of this graph is bounded. 
In particular, this implies that the Euclidean arc-length element along $\{ \delta = B \}$ is comparable to $dx$. 

The form of the WP metric can be derived by invoking Wolpert's estimates \cite{Wol} in the special case of the modular surface. 
By converting \cite[Page 284]{Wol} to the co-ordinates $(x,y)$ up to a uniform homothety, the WP metric is modelled to the first order by 
\begin{equation}\label{e.metric}
\frac{dx^2 + dy^2}{y^3}. 
\end{equation}

Let $I[a,b]$ be the strip $\{ (x, y) \in \mathbb{D} \, : \,  a \leqslant x \leqslant b \}$.
Using the model metric (\ref{e.metric}) and the discussion on the Euclidean arc length element it follows that if $(b-a)$ is large enough (depending only on $B$) then the WP length of $\{ \delta = B \} \cap I[a,b]$ equals $(b-a)/ B^6$ up to a multiplicative constant that depends only on $B$. 
We note that $(b-a)/ B^6$ is also the length of the segment $\{ y = 1/B^2\} \cap I[a, b]$ in the model metric.
On the other hand, the hyperbolic length of $\{ y = 1/ B^2 \} \cap I[a, b]$ equals $(b-a)/B^4$.
Thus, up to a multiplicative constant that depends only on $B$, the WP and hyperbolic lengths of these related horocycle segments are equal, and furthermore we fix $B$ once and for all ahead of time. 
This observation will be crucial in the discussion that precedes Lemma \ref{l.contract}.

\subsection{Projected Paths:} 
Let $\thick$ denote the thick part.
Let $\pi_{\WP}: \mathbb{D} \to \thick$ and $\pi_{\hyp}: \mathbb{D} \to \thick$ be the closest point projections on to the thick part in the corresponding metrics.
For any pair of points $p, q$ in $\thick$ let $\gamma_{\WP}[p,q]$ and $\gamma_{\hyp}[p,q]$ be the WP and hyperbolic geodesics connecting them. 
We call $\pi_{\hyp}(\gamma_{\hyp})$ the \emph{hyperbolic projected path} of $\gamma_{\hyp}$. 
Similarly, we call $\pi_{\WP}(\gamma_{\WP})$ the \emph{WP-projected path} of $\gamma_{\WP}$. 

\begin{proposition}\label{p.quasi}
For any pair of points $p, q$ in $\thick$ the projected paths $\pi_{\hyp}(\gamma_{\hyp}[p,q])$ and $\pi_{\WP}(\gamma_{\WP}[p,q])$ are quasi-geodesics in $\thick$ with the corresponding path metrics.
\end{proposition}

For the hyperbolic metric, this proposition is given in \cite[Lemma 2.1]{Gad-Mah-Tio}. 
Here, we adapt the argument in \cite[Lemma 2.1]{Gad-Mah-Tio} to the WP metric. 
The WP metric on $\mathbb{D}$ is negatively curved and pinched away from 0. 
Hence, it is $\text{CAT}(-\kappa)$ for some $\kappa > 0$ and thus Gromov hyperbolic. 
See \cite[Introduction]{Pol-Wei-Wol}.

Let $H$ be a horoball in the collection $\mathcal{H}$ and let $\pi_H : \mathbb{D} \setminus H  \to \partial H$ be the WP closest point projection.
In a Hadamard manifold with pinched negative curvature, the closest point projection to a horosphere is coarsely contracting. 
See \cite[Section 4]{Far}.
Here we give an analogous result for $\pi_H$ but focus only on showing that $\pi_H$ is coarsely distance non-increasing as that is sufficient for our purposes.

Before we state the projection result in Lemma \ref{l.contract}, we make use of the observation at the end of Section \ref{s.horo}. 
The quotient $H / \pi_1(S)$ is a quasi-convex cusp neighbourhood on the WP surface $X$. 
Let $0 < B' < B$ be the largest constant such that $N' = \{\delta < B' \}$ is contained in $H / \pi_1(S)$. 
The approximations \ref{e.metric}  and \ref{e.distance} imply that if $B$ is small enough then $B' \geqslant B/2$.
Let $H'$ be the lift of $N'$ that is contained in $H$. 
Let $\pi_{H'}: \mathbb{D} \setminus H \to \partial H'$ be the WP closest point projection. 
By the observation at the end of Section \ref{s.horo}, it follows that for any segment $\gamma$ in $\mathbb{D} \setminus H$ the projections $\pi_H(\gamma)$ and $\pi_{H'} (\gamma)$ are coarsely the same, that is, there exists constants $k \geqslant 1$ and $c \geqslant 0$ such that 
\[
\frac{1}{k} \ell_{\partial H} (\pi_H(\gamma)) - c < \ell_{\partial H'} \pi_{H'} (\gamma) < k \ell_{\partial H} (\pi_H(\gamma)) + c
\]
where $\ell_{\partial H}$ and $\ell_{\partial H'}$ denote the lengths in the WP path metrics on $\partial H$ and $\partial H'$ respectively. 

We now show that $\pi_H$ is coarsely distance non-increasing. 

\begin{lemma}\label{l.contract}
The WP closest point projection $\pi_H : \mathbb{D} \setminus H  \to \partial H$ is coarsely distance non-increasing, that is, there exists constants $K > 1, C > 0$ such that for any pair of points $p, q \in \mathbb{D} \setminus H$
\[
d_{\partial H} (\pi_H(p), \pi_H(q)) \leqslant K d_{\WP} (p,q) + C.
\]
\end{lemma} 

\begin{proof}
We give a quick sketch. 
By our preceding observation that the projections $\pi_H$ and $\pi_{H'}$ are coarsely the same, it suffices to prove this property for $\pi_{H'}$. 

First, suppose $p, q$ are distinct points in $\mathbb{D} \setminus H$ such that the WP geodesic segment $[p,q]$ is contained in $\mathbb{D} \setminus H$. 
Let $z$ be the point that completes the cusp in $H$.
Let $[p,z]$ and $[q,z]$ be WP geodesics from $p$ and $q$ to $z$. 

As $B$ was chosen small enough for the metric near the cusp to approximate the model metric (\ref{e.metric}) closely, 
the segments $[p,z]$ and $[q,z]$ intersect $\partial H'$ at almost right angles.
This implies that there exists a constant $c' \geqslant 0$ that depends only on the metric and the constant $B$ such that $\pi_{H'}(p)$ is $c'$-close to the point $p' = [p,z] \cap \partial H'$. 
Similarly $\pi_{H'} (q)$ is $c'$-close to the point $q' = [q,z] \cap \partial H'$.

The triangle with vertices $z, p, q$ can now be compared to the comparison triangle in constant negative curvature $-\kappa$ to conclude that the length of $[p' q']$ along $\partial H'$ is strictly less than $d_{\WP} (p,q)$. 
We deduce that the length of $\pi' [p,q]$ along $\partial H'$ is coarsely less than $\ell_{\WP} [p,q]$.

Finally, any path in $\mathbb{D} \setminus H$ can be closely approximated by a finite concatenation of geodesic segments in $\mathbb{D} \setminus H$. 
The lemma follows.
\end{proof} 

\begin{proof}[Proof of Proposition~\ref{p.quasi} for the WP metric] 
Suppose that the number of horoballs in $\mathcal{H}$ that the geodesic segment $\gamma_{\WP}[p,q]$ passes through is  $k$. 
We decompose $\pi_{\WP}(\gamma_{\WP})$ into segments $[y_0, y_1] \cup [y_1, y_2] \cup [y_2, y_3] \cup \dots \cup [y_{2k-1}, y_{2k}] \cup [y_{2k}, y_{2k+1}]$, where $y_0 =$, $y_{2k+1} =q$, all intervals of the form $[y_{2j}, y_{2j+1]}$ are WP geodesic segments in $\thick$ and the remaining intervals are segments along distinct horocycles. 

Let $\beta_j$ be the WP geodesics perpendicular to $\gamma_{\WP}[y_0, y_{2k+1}]$ at the points $y_j$. 
Each $\beta_j$ separates $y_0$ from $y_{2k+1}$ so a thick geodesic from $y_0$ to $y_{2k+1}$ must intersect $\beta_j$. Moreover, the geodesics $\beta_j$ are all disjoint. 
Otherwise,  for some $j$ we would find a triangle bounded by the geodesic segments of $\beta_j, \beta_{j+1}$ and $\gamma_{\WP}[y_j, y_{j+1}]$. 
As the WP metric has pinched negative curvature, this violates the Gauss-Bonnet Theorem as the sides are geodesic segments and two angles of the triangle are right angles. 

In conclusion, the geodesics $\beta_j$ divide $\mathbb{D}$ into regions $R_j$ such that each $R_j$ contains a single subsegment of $\gamma_{\WP} [y_0, y_{2k+1}]$ that is either entirely contained in $\thick$ or entirely contained in a horocycle. 
In each $R_j$, we now show that the subsegment of the projected path is coarsely the shortest path in the induced path metric, that is any thick geodesic segment that makes it across $R_j$ between the endpoints has length coarsely larger than the subsegment of the projected path of $\pi_{\WP}(\gamma_{\WP})$ in $R_j$.

First consider a region $R_j$ where $[y_{2j}, y_{2j+1}]$ is a WP geodesic segment in $\thick$. 
As the metric is $\text{CAT}(-\kappa)$ for some $\kappa > 0$, the nearest point projection to a geodesic is coarsely contracting. 
Thus the length of any path that crosses from $\beta_{2j}$ to $\beta_{2j+1}$ is coarsely at least the length of $[y_{2j}, y_{2j+1}]$. 

Now consider a region $R_j$ where $[y_{2j-1}, y_{2j}]$ is a subsegment of a horocycle $\partial H$.
By lemma \ref{l.contract}, the closest point projection $\pi_H: \thick \to \partial H$ is also coarsely contracting. 
Thus the length of any thick path that crosses $R_j$ from $\beta_{2j-1}$ to $\beta_{2j}$ is coarsely at least the length $\ell_{\partial H}(y_{2j-1}, y_{2j})$.

This finishes the proof of the proposition for the WP metric.

\end{proof}

The fundamental group $\pi_1(S)$ acts co-compactly on $\thick$.
Hence, by the Svar\'{c}-Milnor lemma, $\thick$ with either WP or hyperbolic induced path metric is quasi-isometric to the fundamental group $\pi_1(S)$ with a word metric given by any finite generating set. 
In particular, $(\thick, d_{\WP})$ and $(\thick, d_{\hyp})$ are quasi-isometric.
As a further corollary of this quasi-isometry and of Proposition \ref{p.quasi}, we get

\begin{corollary}\label{c:fellow}
There exists a constant $R > 0$ such that for any pair of points $p,q$ in $\thick$ the segments $\pi_{\WP}[p,q]$ and $\pi_{\hyp}[p,q]$ $R$-fellow travel in $\thick$. 
\end{corollary} 

\begin{proof}
By \cite[Lemma 2.1]{Gad-Mah-Tio}, the projected path $\pi_{\hyp}[p,q]$ is a quasi-geodesic in $(\thick, d_{\hyp})$, and by Proposition \ref{p.quasi}, the projected path $\pi_{\WP}[p,q]$ is a quasi-geodesic in $(\thick, d_{\WP})$. As $(\thick , d_{\WP})$ and $(\thick, d_{\hyp})$ are quasi-isometric, both projected paths $\pi_{\WP}[p,q]$ and $\pi_{\hyp}[p,q]$ being quasi-geodesics, fellow travel.

\end{proof}

\subsection*{Proof of Theorem \ref{t.main}:}
By Corollary \ref{c:fellow}, the projected paths $\pi_{\WP}[p,q]$ and $\pi_{\hyp}[p,q]$ $R$-fellow travel. This implies that there is a constant $R'> 0$ such that if $\pi_{\WP}(\gamma_{WP})$ has a subsegment of length at least $R'$ along a horocycle $\partial H$ then $\pi_{\hyp}(\gamma_{hyp})$ must also contain a subsegment of $\partial H$, and vice versa. 
Moreover, the respective entry and exit points in $\partial H$ for the two projected paths must be within distance $R$ of each other in $\thick$ with the path metric. 
It is possible that if $\pi_{\WP}(\gamma_{WP})$ has a subsegment along $\partial H$ that is less than $R'$ in length then $\gamma_{\hyp}$ does not enter $H$ at all, and vice versa. 
If such is the case then we choose the sub-segment along $\gamma_{\hyp}$ that records the excursion in $H$ to be empty, and vice versa. 
This concludes the proof that for any pair of points $p,q \in \thick$, the respective geodesics $\gamma_{\WP}[p,q]$ and $\gamma_{\hyp}[p,q]$ fellow travel in the thick part and all deviations between them arise in cusp neighbourhoods. 

\section{Deviations during excursions}

We now quantify deviations during a single excursion. 
Consider a cusp excursion of a WP geodesic during which the minimum distance to the cusp  satisfies $\delta_{\textrm{min}} = 1/D$ for $D > 1$. 

In upper half space co-ordinates, a WP geodesic during an excursion describes a curve $\gamma_{\WP}(t)= (x(t), y(t))$ whose tangent vector $v(t)$ has (up to a uniform multiplicative constant) unit size with respect to the model metric (\ref{e.metric}).  
For convenience, we apply a homothety whose constant depends only on $B$ to assume that we are working with $\{z \in \mathbb{C} \, : \, \text{Im} z  \geqslant 1\}$ as the cusp neighbourhood.
This means that all our estimates below are correct up to a multiplicative constant that depends only on $B$. 

An excursion in which $\delta_{\textrm{min}} = 1/D$ with $D > 1$, corresponds to $y_{\text{max}} = D^2$. 
By \cite[Proposition 5.5]{Gad-Mat}, the winding number about the cusp for such an excursion is $D^2$ up to a uniform multiplicative constant. 
In particular, this implies that during such an excursion the point that distance along $\partial H = \{ z \, : \, \text{Im} z = 1\}$ between the entry and exit points is $D^2$ up a uniform multiplicative constant. 

We want to compare $\gamma_{\WP}$ with a hyperbolic geodesic $\gamma_{\hyp}$ that enters and exits $\{z \in \mathbb{C}: \, \mid \text{Im} z \mid \geqslant 1 \}$ within bounded hyperbolic distance of the $\gamma_{\WP}$-entry and exit points.  So we may assume that in the model co-ordinates the entry and exit points of $\gamma_{\WP}$ are $(0,1)$ and $(2D^2,1)$ respectively.

We will analyse the two trajectories $\gamma_{\WP}$ and $\gamma_{\hyp}$ over the first half of their excursions, that is between the entry point and the point with the largest imaginary part $y_{\textrm{max}}$. 
By symmetry, the same analysis holds for the trajectories on their way out. 
For the first half, the hyperbolic geodesic is given by the locus $= \left(x,\sqrt{1+ 2D^2  x - x^2} \right)$. 
To prove our deviation bounds, we will similarly write down the locus for $\gamma_{\WP}$ and then compare with above locus.

\begin{lemma}\label{l.deviations}
A point on the ingoing half of $\gamma_{\WP}$ is given in co-ordinates by
\[
\asymp \left( x, D^{6/5} x^{2/5} \right)
\]
where $\asymp$ means up the individual co-ordinates may differ up to a uniform multiplicative constant. 
\end{lemma}

\begin{proof}
The co-ordinate $x(\delta)$ is the winding number when $\gamma_{\WP}$ is distance $\delta$ from the cusp. In the calculation using dyadic intervals in the proof of \cite[Proposition 5.5]{Gad-Mat}, let $k$ be such that
\[
c 2^k \asymp \frac{1}{D^3 \delta^3}
\]
where $c = 1/(D^3 \delta_0^3)$. Recall that $J$ was chosen so that
\[
c 2^J \asymp 1
\]
Then repeating line by line the calculation in \cite[Proposition 5.5]{Gad-Mat}, we have
\[
x(\delta) \asymp D^3 \sum_{j=1}^k (c2^j)^{5/3} \asymp D^2 (c 2^k)^{5/3} \asymp \frac{1}{D^3 \delta^5}
\]
On the other hand,
\[
y(\delta) \asymp \frac{1}{\delta^2} 
\]
By expressing $y$ in terms of $x$, we get that a point on the ingoing trajectory of $\gamma_{\WP}$ has coordinates
\[
\asymp \left( x, D^{6/5} x^{2/5} \right)
\]
finishing the proof.
\end{proof}

It is interesting to compare the \texttt{"}shapes\texttt{"} of the hyperbolic and WP trajectories that we considered above.
A hyperbolic geodesic $\gamma_{\hyp}$ that has coarsely the same entry and exit points, namely $(0,1)$ and $(2D^2, 1)$, also attains the maximal imaginary part $D^2$, that is, it is about the same height in the upper half space co-ordinates as $\gamma_{\WP}$. 
However, from Lemma \ref{l.deviations}, we observe that quantitatively the WP trajectory has faster initial gradient. 
So it goes deep in to the cusp neighbourhood more quickly and winds thereafter.
This is consistent with the standard intuition of how geodesics wind around the cusp in the WP metric.

\section{The circle map} 

We fix a base-point $p \in \mathbb{D}$.
By shrinking the cusp neighbourhoods if required, we may assume that $p$ lies in the thick part.
By the results in \cite{Pol-Wei}, the WP geodesic flow is ergodic. 
This implies that a dense set of directions $v \in T^1_p (\mathbb{D}, d_{\WP})$ give WP rays that are recurrent to the thick part.
In fact, identifying $T^1_p (\mathbb{D}, d_{\WP})$ with $S^1$, the set of recurrent directions has full measure in the Lebesuge measure on $S^1$.

For a recurrent direction $v$, let $\gamma^v_{\WP}(t)$ denote the recurrent WP ray that it defines. 
Let $t_n$ be any monotonic sequence of recurrence times such that $t_n \to \infty$. 
We consider the sequence of hyperbolic geodesic segments $[x, \gamma^v_{\WP}(t_n)]$. 
By Theorem \ref{t.main}, the hyperbolic and WP projected paths fellow travel. 
It follows that hyperbolic geodesic segments $[x, \gamma^v_{\WP} (t_n)]$ fellow travel.
So we can pass to a limit to get a unique limiting hyperbolic ray $\gamma^v_{\hyp}$ that fellow travels $\gamma^v_{\WP}$ in $\thick$.
In particular, the hyperbolic ray $\gamma^v_{\hyp}$ is recurrent to $\thick$.
Identifying $T^1_p (\mathbb{D}, d_{\hyp})$ also with $S^1$, we get a map $\psi$ from a dense (full measure) subset of $S^1$ to $S^1$ by sending $v$ to the unit tangent vector at $p$ to $\gamma^v_{\hyp}$. 

In similar fashion, we get the reverse map $\xi$ from a dense (full measure) subset of $S^1$ to $S^1$.
For a recurrent hyperbolic ray $\zeta$ determined by $v \in T^1_p (\mathbb{D}, d_{\hyp})$, we can consider a monotonic sequence of recurrence times $s_n$ such that $s_n \to \infty$ and consider the WP geodesic segments $[x, \zeta(s_n)]$. 
As $(\mathbb{D}, d_{\WP})$ is $\text{CAT}(-\kappa)$ for some $\kappa > 0$, the longer and longer WP geodesic segments fellow travel.
Thus, we can pass to a limiting WP ray. 
By setting up identifications with $S^1$ and sending $v$ to the unit tangent base vector for this WP ray, we define the map $\xi$ on a dense subset of $S^1$ to $S^1$. 

By Theorem \ref{t.main}, the WP ray defined by the direction $\xi \circ \psi(v)$ fellow travels in $\thick$ the WP ray defined by $v$. 
As $(\mathbb{D}, d_{\WP})$ is $\text{CAT}(-\kappa)$ for some $\kappa> 0$, it follows that the ray defined by $\xi \circ \psi(v)$ has to fellow travel the ray defined by $v$ for all times. 
Hence, the rays are identical and $\xi \circ \psi = \text{id}$.
By a similar reasoning, the composition $\psi \circ \xi$ is also well defined and identity. 
Thus, $\psi$ is a bijection from the dense set of recurrent WP-directions to the dense set of recurrent hyperbolic directions.

\subsection{Monotonicity:} We now derive the monotonicity of the map $\psi$. 
First, we fix some notation. 
Given $\alpha, \beta \in S^1$, we denote by $I(\alpha, \beta)$ the (open) sector of directions in $S^1$ going counter-clockwise from $\alpha$ to $\beta$.
We identify the unit tangent spaces $T^1_p (\mathbb{D}, d_{\WP})$ and $T^1_p (\mathbb{D}, d_{\hyp})$ with $S^1$ in way that preserves the orientation induced by the surface $S$.

\begin{lemma}\label{l.circular}
Suppose $v, w \in  \text{Domain} (\psi)$ and distinct. 
For any $\theta \in I(v,w) \cap \text{Domain} (\psi)$, we have $\psi(\theta) \in I(\psi(v), \psi(w)) $.
\end{lemma}

\begin{proof}
The proof follows the same Gauss-Bonnet type arguments. Suppose that there is some $\theta \in I(v, w)$ such that $\psi(\theta)$ is not contained in $I(\psi(v), \psi(w))$. 

Let $S(\alpha,\beta)$ be the open set in $(\mathbb{D}, d_{\hyp})$ swept out by hyperbolic-rays with initial vectors in $I(\alpha, \beta)$. 
It is bounded by the hyperbolic rays $\gamma^\alpha_{\hyp}$ and $\gamma^\beta_{\hyp}$. 

As $\psi(\theta)$ is not in $I(\psi(v), \psi(w))$, the hyperbolic ray $\gamma^{\psi(\theta)}_{\hyp}$ does not intersect $S(\psi(v) ,\psi(w))$. 
Let us fix a monotonic sequence $t_n$ of recurrence times along $\gamma^{\psi(\theta)}_{\hyp}$ such that $t_n \to \infty$. 
Suppose that the $R$-ball centred at some $\gamma^{\psi(\theta)}_{\hyp}(t_n)$ is disjoint from the rays $\gamma^{\psi(v)}_{\hyp}$ and $\gamma^{\psi(w)}_{\hyp}$. 
As $\gamma^\theta_{\WP}$ intersects this ball it follows that $\gamma^{\theta}_{\WP}$ must intersect either $\gamma^v_{\WP}$ or $\gamma^w_{\WP}$. 
Breaking symmetry, let us assume that $\gamma^{\theta}_{\WP}$ intersects $\gamma^v_{\WP}$.
In this case the WP rays $\gamma^{\theta}_{\WP}$ and $\gamma^v_{\WP}$ bound a bigon, contradicting Gauss-Bonnet.

So now suppose that for all $t_n$, the $R$-balls centred at $\gamma^{\psi(\theta)}_{\hyp}(t_n)$ intersect one of $\gamma^{\psi(v)}_{\hyp}$ or $\gamma^{\psi(w)}_{\hyp}$. 
Breaking symmetry, let us assume that it is $\gamma^{\psi(v)}_{\hyp}$.
As this implies the ray $\gamma^{\psi(\theta)}_{\hyp}$ fellow travels the ray $\gamma^{\psi(v)}_{\hyp}$ for all time, it forces $\gamma^{\psi(\theta)}_{\hyp} = \gamma^{\psi(v)}_{\hyp}$.
It follows that $\gamma^\theta_{\WP}$ fellow travels $\gamma^v_{\WP}$.
As the WP metric is $\text{CAT}(-\kappa)$ for some $\kappa> 0$ and $(\mathbb{D}, d_{\WP})$ is a geodesic metric space, it forces $\gamma^\theta_{\WP} = \gamma^v_{\WP}$.
But this is a contradiction because $\theta \in I(v,w)$. 

\end{proof}

Lemme \ref{l.circular} implies monotonicity for the map $\psi$. 
In summary, $\psi$ is a map defined on a dense subset of $S^1$, has dense image, and preserves the circular order.
By a standard exercise, see \cite[Lemma 7.3]{Dov-Mar-Rya-Vuo}, $\psi$ extends to a $S^1$-homeomorphism.

\section{Typical geodesics}\label{s.random}

Let $\ell$ be a Lebesgue measure on $S^1$. 
After identification with $T_p^1 (\mathbb{D}, d_{\WP})$, we can let $\ell$ be the conditional measure induced by the WP Liouville measure. 
Alternatively, $\ell$ could also be the conditional measure induced by the hyperbolic Liouville measure because the two densities are absolutely continuous. 

\subsubsection*{\bf Proof of Theorem \ref{t.singular}:}
We derive Theorem \ref{t.singular} by comparing the growth of the total winding number about the cusp along typical WP and hyperbolic geodesics. 
For either metric, we define the total winding number till a particular recurrence time to be the sum of the winding numbers (without sign) during cusp excursions till that time.

The statistics of cuspidal winding numbers along typical WP geodesics for general WP-type metrics were analysed in \cite{Gad-Mat}. 
The analysis heavily leverages the precise decay of correlations proved in \cite{BMMW} for the exponential mixing of the WP flow. 

For the total winding number comparison to be valid, we also need the time changes between the WP and hyperbolic times to work out.
To be precise, along typical geodesics the WP and hyperbolic times at recurrence points need to be asymptotically comparable.
This is also established in \cite{Gad-Mat} for general WP-type metrics and we shall recall the relevant results at the appropriate steps in the proof of Theorem \ref{t.singular}. 
As the arguments below apply to mildly constrained WP-type metrics, we keep the discussion on a a general footing as before. 

By \cite[Theorem 5.9]{Gad-Mat}, there exists a constant $K_1 > 1$ such that for almost every $p$ and $\ell$-almost every $v \in T_p^1 (\mathbb{D}, d_{\WP})$, the total winding number $W_v(T)$ of the Weil--Petersson ray defined by $v$ around the cusp till time $T$, satisfies
\begin{equation}\label{e.wp-wind}
\frac{1}{K_1} T < W_v(T) < K_1T.
\end{equation}
Also by \cite[Section 5.10]{Gad-Mat}, the hyperbolic distance measured along the recurrence times of a recurrent Weil--Petersson geodesic ray increases linearly in time: there exists a constant $K_2 > 1$ such that 
\begin{equation}\label{e.time-change}
\frac{1}{K_2} T < d_{hyp}(p, \gamma^v_{\WP}(T))  < K_2 T.
\end{equation}
In particular, this implies that when $T$ is a time of recurrence along $\gamma^v_{\WP}$ the hyperbolic time $S$ along $\gamma^{\psi(v)}_{\hyp}$ is comparable to the Weil--Petersson time $T$.
Let $W_{\psi(v)}(S)$ be the total winding number around the cusp till time $S$ of the hyperbolic ray $\gamma^{\psi(v)}_{\hyp}$. 
From Theorem \ref{t.main} it follows that there exists a constant $K_3 > 1$ such that when $S$ is the hyperbolic time that corresponds to a time of recurrence $T$ along $\gamma^v_{\WP}$, we have
\begin{equation}\label{e.wind-rel} 
\frac{1}{K_3} W_v(T)  < W_{\psi(v)}(S) < K_3 W_v(T).
\end{equation} 
Putting inequalities \ref{e.wp-wind}, \ref{e.time-change} and \ref{e.wind-rel} together, we can conclude that there exists a constant $K_4> 1$ such that
\begin{equation}\label{e.winding}
\frac{1}{K_4} S < W_{\psi(v)}(S) < K_3 S
\end{equation} 
i.e., the total winding along the corresponding hyperbolic geodesic grows linearly in hyperbolic time.

On the other hand, by \cite[Theorem 1.6]{Gad}, for $\ell$-almost every $v \in (T_p^1 \mathbb{D}, d_{\hyp})$, there is $S \log S$ lower bound for the growth of total winding number till time $S$ along the hyperbolic geodesic ray $\gamma^v_{\hyp}$. 
Thus, the hyperbolic geodesic rays that are the images of typical WP rays under $\psi$ are atypical and this concludes the proof of Theorem \ref{t.singular}.

\subsection{Continued fractions:} 
We can now invoke Theorem \ref{t.singular} for the modular surface to give a concrete description. 
Series showed that hyperbolic geodesic rays on the modular surface are coded by the continued fraction expansion of their endpoint at infinity. 
See \cite{Ser}.
The continued fraction coefficients track the combinatorics of how the ray crosses the triangles in the Farey tessellation.
Namely, for each triangle that the geodesic cuts through one records right or left depending on which of the resulting components is again a triangle.
In this way, one gets an infinite sequence of $R$s and $L$s.
Breaking symmetry, if the sequence has the form $R^{a_1} L^{a_2} \cdots $ then the number $a_j$ is exactly the $j^{\text{th}}$ coefficient of the endpoint at infinity. 
Suppose that $T$ is a recurrence time, $N(T)$ is the number of excursions till $T$, and $H_j \, : \, 1 \leqslant j \leqslant N(T)$ is the horoball for the $j^{\text{th}}$-excursion.
We may choose upper half-space co-ordinates $(x,y)$ so that $H_j$ is of the form $\{ y > c \}$ for some constant $c \geqslant 1$ and the fundamental domain intersected with $H_j$ is the strip $0 \leqslant x < 1$. 
Then it readily follows that up to a multiplicative constant that depends only on $c$, the distance between the entry and exits points on $\partial H_j$ equals the continued fraction coefficient $a_j$. 
See \cite{Gad} for the details.

Using the fellow travelling result, namely Theorem \ref{t.main}, we can also record continued fraction coefficients along a recurrent WP ray $\gamma^v_{\WP}$ by tracking the coefficients along the corresponding hyperbolic ray $\gamma^{\psi(v)}_{\WP}$.
We could also define coefficients using the number of cuspidal fundamental domains crossed during each excursion of $\gamma^v_{\WP}$.
Again by Theorem \ref{t.main}, the two definitions produce the same coefficients that differ up to a bounded additive error. 
So from now on, we will use the second definition. 
 
Consider a recurrence time $T$ for $\gamma^v_{\WP}$. 
Then the segment $[\gamma^v_{\WP}(0), \gamma^v_{\WP}(T)]$ has completed finitely many excursions giving us a finite list of coefficients $[a_1, a_2, \dots, a_{N(T)}]$. 
Note that there is a constant $K_5 > 1$ such that for any recurrence time $T$,
\[
\frac{1}{K_5} W_v(T) <  a_1 + a_2 + \dots + a_{N(T)} < K_5 W_v(T)
\]
where $W_v(T)$ is the total cuspidal winding number till time $T$. 
By the ergodicity of the Weil--Petersson flow and an argument similar to \cite[Lemma 3.8]{Gad}, there exists a constant $K_6 > 1$ such that 
\[
\frac{1}{K_6} T < N(T) < K_6 T,
\]
that is, the number of excursions till time $T$ along a typical WP ray grow linearly in $T$. 
In particular, combining the above facts about winding number growth and the number of excursions with (\ref{e.wp-wind}), we get that there exists a constant $K_7 > 1$ such that along a typical WP ray $\gamma^v_{\WP}$, we have 
\[
\frac{1}{K_7} < \frac{a_1 + a_2 + \dots + a_{N(T)}}{N(T)} < K_7
\]
for all recurrence times $T$ that are large enough (how large depends on $v$).
In other words, the average coefficient along a typical WP ray (and hence along the corresponding hyperbolic ray) is bounded. On the other hand, the average coefficient along a typical hyperbolic ray is unbounded and exhibits an interesting \texttt{"}trim-sum\texttt{"} property. See the introduction in \cite{Gad}.


\end{document}